\documentclass{amsart}
\usepackage[left=1in,right=1in,top=1.2in,bottom=1in]{geometry}
\usepackage[utf8]{inputenc}
\usepackage{setspace}
\usepackage{amsmath,amssymb}
\usepackage{comment}
\usepackage{xcolor}
\usepackage{url}
\usepackage[colorlinks=true,allcolors=blue,backref=page]{hyperref}
\usepackage[noabbrev,capitalize,nameinlink]{cleveref}
\usepackage{bbm}
\usepackage{tikz}

\newtheorem{theorem}{Theorem}

\newtheorem{proposition}[theorem]{Proposition}
\newtheorem{lemma}[theorem]{Lemma}

\newtheorem{corollary}[theorem]{Corollary}

\newtheorem*{claim*}{Claim}

\theoremstyle{remark}
\newtheorem*{remark*}{Remark}

\newcommand{\Xt}{\widetilde{X}}

\newcommand{\cE}{\mathcal{E}}
\newcommand{\cG}{\mathcal{G}}

\newcommand{\hypergeom}{\operatorname{HyperGeom}}
\newcommand{\Var}{\operatorname{Var}}
\newcommand{\Z}{\mathbb{Z}}

\newcommand{\eps}{\varepsilon}
\renewcommand{\epsilon}{\eps}
\renewcommand{\P}{\mathbb{P}}
\newcommand{\E}{\mathbb{E}}
\newcommand{\one}{\mathbbm{1}}

\begin{document}

\title[Comparability of random permutations in the strong Bruhat order]{Comparability of random permutations in the strong Bruhat order}
\author{Nicholas Christo}

\address{University of Illinois Chicago. Dept of Mathematics, Statistics and Computer science.}
\email{nchrist5@uic.edu}

\author{Marcus Michelen}
\address{Northwestern University, Dept of Mathematics.}
\email{michelen@northwestern.edu}

\begin{abstract}
    The (strong) Bruhat order for permutations provides a partial ordering defined as follows: two permutations are comparable if one can be obtained from the other by a sequence of adjacent transpositions that each increase the number of inversions by $1$.  Given two random permutations, what is the probability that they are comparable in the Bruhat order?  This problem was first considered in a 2006 work of Hammett and Pittel, which showed an exponential lower bound and a polynomial upper bound.   The lower bound was very recently improved to the subexponential bound of $\exp(-n^{1/2 + o(1)})$ by  Boretsky, Cornejo, Hodges, Horn, Lesnevich, and McAllister.  Hammett and Pittel predicted that the probability should decrease polynomially. We show that the probability decreases faster than any polynomial and is on the order of $\exp(-\Theta(\log^2 n))$.
\end{abstract}

\maketitle

\section{Introduction}

The \emph{(strong) Bruhat order} is a partial order in the symmetric group: two permutations are comparable if one can be obtained from the other by a sequence of adjacent transpositions that increase the number of inversions by $1$.  An easy-to-check criterion for comparability is the so-called $(0,1)$-matrix criterion: given permutations $\pi$ and $\tau$, let $M_{\pi}$ and $M_\tau$ be the corresponding permutation matrices.
Then the Bruhat order $\leq$ is defined via \cite[Thm.~2.1.5]{bjorner2005combinatorics}
\begin{equation}\label{eq:comparability-definition}
    \pi \leq \tau \iff \sum_{i \leq a,\, j\leq b} M_{\tau}(i,j) \leq \sum_{i \leq a,\, j \leq b} M_\pi(i,j) \quad \text{ for all } a, b \leq n\,.
\end{equation}

The Bruhat order is a central object in algebraic combinatorics and plays a leading role in the study of Schubert varieties.  This line of work began with Ehresmann's 1934 work \cite{ehresmann1934topologie} which included a characterization of the Bruhat order in terms of the increasing rearrangements of $\pi$ and $\tau$ that is equivalent to \eqref{eq:comparability-definition}.  There are many further equivalent descriptions of the Bruhat order \cite{ Bjorner1984, deodhar1977some, drake2004two, fulton1997young,  lascoux1996treillis} including criteria of a similar combinatorial flavor to \eqref{eq:comparability-definition} and of a more algebraic flavor.

In this note, we study the probability that two random permutations are comparable.  In particular, if we take $\pi$ and $\tau$ to be independent and uniformly chosen from the symmetric group $S_n$, what is the probability that $\pi \leq \tau$?  This problem was introduced in a 2006 work of Hammett and Pittel \cite{hammett2008often} which showed \begin{equation*}
    c(.708)^n \leq \P(\pi \leq \tau) \leq C n^{-2}
\end{equation*}
for constants $C,c > 0$.  
Very recently\footnote{The work \cite{boretsky2026comparability} appeared on arXiv while the present work was nearing completion.  The main focus of the work \cite{boretsky2026comparability} is showing that if one considers a different order on $S_n$ known as the \emph{weak Bruhat order} $\leq_W$ then $\P(\pi \leq_W \tau) = \exp(-(1/2 + o(1))n\log n)$.}, Boretsky-Cornejo-Hodges-Horn-Lesnevich-McAllister \cite{boretsky2026comparability} improved the lower bound to the form $\exp(-c\sqrt{n} \log^{3/2}n)$. In their work, Hammett and Pittel stated that ``Empirical estimates...suggest that $\P(\pi \leq \tau)$ is of order $n^{-(2+\delta)}$ for $\delta$ close to $0.5$.''   Our main theorem shows that in fact the probability decreases faster than any polynomial at the rate $\exp(-\Theta(\log^2 n))$: \begin{theorem}\label{th:main}
    Let $\pi,\tau \in S_n$ be independently and uniformly chosen at random.  There are constants $C,c> 0$ so that for $n$ sufficiently large we have $$ \exp\left(- C \log^2 n \right) \leq \P(\pi \leq \tau) \leq \exp\left(- c \log^2 n \right)\,.$$
\end{theorem}

The main idea is to define $$Z(a,b) = \sum_{i\leq a,\, j \leq b} \left(M_\pi(i,j) - M_\tau(i,j)\right)$$ then \eqref{eq:comparability-definition} may be rewritten as \begin{equation}\label{eq:intro-persistence}
    Z(a,b) \geq 0 \text{ for all }a, b \leq n\,.
\end{equation}

We treat \eqref{eq:intro-persistence} as a two-dimensional persistence event.  As a useful point of comparison, a classical persistence event is asking for the probability that a simple random walk is non-negative for the first $n$ steps.  A simple random walk has only one time index, while \eqref{eq:intro-persistence} has two, and so there is not as simple of a combinatorial approach as in the case of a simple random walk.  However, two-dimensional persistence problems for Gaussian processes---such as the Brownian sheet---have received a fair amount of attention.  Our approaches to the upper and lower bounds are inspired by these methods from probability theory, and in fact our proof of the upper bound uses some of the machinery for analyzing Gaussian processes to bound our problem \eqref{eq:intro-persistence} in terms of an analogous Gaussian problem.

We prove the upper and lower bounds separately in \cref{sec:UB} and \cref{sec:LB} as \cref{pr:upper-bound} and \cref{prop:LB} respectively.  We begin with a sketch of our arguments along with a more detailed conjecture about the behavior of $\P(\pi \leq \tau)$ (see \eqref{eq:conjecture}).

\subsection{Proof outline}

To first see heuristically that $\exp(-\Theta(\log^2n ))$ is the correct probability for the event in \eqref{eq:intro-persistence}, note that $Z(a,b)$ is a mean-zero random variable with variance $\approx ab/n$.  Further, while $Z(a,b)$ is not strictly a sum of independent random variables, one may show that $Z(a,b)$ obeys a central limit theorem and so is close to a Gaussian random variable of variance $\approx ab/n$.  As a consequence, $\P(Z(a,b) \geq 0) \approx 1/2$ for each fixed $(a,b)$ (once $ab/n \gg 1)$.  When $(a_1,b_1)$ and $(a_2,b_2)$ are close together, the random variables $Z(a_1,b_1)$ and $Z(a_2,b_2)$ are quite correlated.  The idea is to find a set of $(a,b)$ so that all pairwise correlations are uniformly bounded away from $1$.  In particular, if we consider $\mathcal{S}_0 = \{\rho^i : i \in \{0,1,2,\ldots \}\} \cap [n]$ for some fixed integer $\rho \geq 2$, then for all pairs $(a_1,b_1),(a_2,b_2) \in \mathcal{S}_0^2$ with $(a_1,b_1) \neq (a_2,b_2)$ we have \begin{equation}\label{eq:upper-bound-corr-intro}
    \frac{|\E[Z(a_1,b_1) Z(a_2,b_2)]|}{\sqrt{\Var(Z(a_1,b_1)) \Var(Z(a_2,b_2))}} \leq 1 - \delta
\end{equation}
for some $\delta >0$ depending on $\rho$.  Here one might imagine that across all $(a,b) \in \mathcal{S}_0^2$, the random variables $(Z(a,b))$ are sufficiently independent so that \begin{equation}\label{eq:factor-intro}
\P(Z(a,b) > 0 \text{ for all }(a,b) \in \mathcal{S}_0^2) \approx \prod_{(a,b) \in \mathcal{S}_0^2} \P(Z(a,b) > 0) = \exp(-\Theta( |\mathcal{S}_0^2|)) = \exp(-\Theta(\log^2n)) \,.\end{equation}

While such a strong approximate independent statement does not precisely hold, this heuristic provides the correct shape of the proof of both the upper bound and lower bound, whose technical details are quite different.  

For the upper bound, it turns out that if $(Z(a,b))_{(a,b) \in \mathcal{S}_0^2}$ is a Gaussian vector, then an assumption similar to \eqref{eq:upper-bound-corr-intro} \emph{does} imply an upper bound similar to \eqref{eq:factor-intro}, with a loss that is exponential in $|\mathcal{S}_0|^2.$  This is due to a theorem of Li and Shao \cite{li2004lower} which we restate in \cref{th:general-gaussian-upper-bound}.  However, it is a highly nontrivial task to approximate $(Z(a,b))_{(a,b) \in \mathcal{S}_0^2}$ with Gaussians simultaneously, especially since $Z(a,b)$ are not sums of independent and identically distributed (i.i.d.) random variables. 

Our first step towards a Gaussian comparison is to show that if we only consider submatrices of $M_\pi(i,j)$ up to $N = \lfloor n^{7/12} \rfloor,$ then the distribution of $M_\pi$ is essentially the same as an $N \times N$ random matrix whose entries are i.i.d.\ $\mathrm{Bernoulli}(1/n)$ variables. In order to understand our choice of $N$, note that the expected number of $1$'s in the first $N \times N$ submatrix of $M_\pi$ is $N^2 / n$.  If we hope to show that \eqref{eq:intro-persistence} is rare, we need to have that the number of $1$'s is diverging, and so we require $N \gg n^{1/2}$.  Note however that in an $N \times N$ matrix with $\mathrm{Bernoulli}(1/n)$ entries, the number of rows or columns with at least two $1$'s is of order
$$ \approx N \cdot(N / n)^2 = N^3 / n^2$$
and so we need $N \ll n^{2/3}$.   Balancing the requirements of $N \gg n^{1/2}$ and $N \ll n^{2/3}$ yields the choice of $N = \lfloor n^{7/12} \rfloor$.  We show that we may replace an $N \times N$ submatrix of $M_\pi$ and $M_\tau$ with i.i.d.\ $\mathrm{Bernoulli}(1/n)$ entries, which thus allows us to view $Z(a,b)$ as a sum independent random variables.  This is performed in \cref{lem:Bernoulli-comparison}. 

Finally, we lean on a strong approximation theorem of Rio \cite{rio1995strong} (see \cref{th:KMT}) which allows us to replace the underlying random variables with standard Gaussian random variables, once we group enough random variables together to attain variance of order $1$.  The work of Rio \cite{rio1995strong} may be viewed as a two-dimensional analogue of the Koml\'os-Major-Tusnady (KMT) coupling, which shows that a random walk and Brownian motion may be coupled together so the entire trajectories remain close.  Altogether, this completes the upper bound.

For the lower bound, we again will consider dyadic scales.  However, we have an additional tool that is helpful for showing a version of \eqref{eq:factor-intro}: each event $Z(a,b) > 0$ is an \emph{increasing} event (in the Bruhat order) in the permutation $\pi$ and \emph{decreasing} in $\tau$.  A correlation inequality of Johnson-Leader-Long \cite{johnson2020correlation} (see \cref{thm:JLL-FKG}) states that such events are \emph{positively correlated} in the sense that $\P(A \cap B) \geq \P(A) \P(B).$ This may be understood as a version of the Fortuin-Kasteleyn-Ginibre (FKG) inequality for the Bruhat order.  

Our first use of this correlation inequality is to show $$\P(Z(a,b) \geq 0 \text{ for all }a,b \leq n) \geq \P(Z(a,b) \geq 0\text{ for all }a,b \leq \lceil n/2 \rceil)^4$$
(see \cref{lem:E-equiprobable}).  For a lower bound on the resulting probability, we will give ourselves $\Omega(\log n)$ of wiggle room at a cost of $\exp(-\Theta(\log^2n))$ by conditioning on the first and last $\Theta(\log n)$ entries of the permutations underlying $Z$.  After breaking into dyadic rectangles and iteratively using the correlation inequality \cref{thm:JLL-FKG}, we see it is enough to show 
    \begin{equation*}
        \P\left(\min_{(a,b) \in [x,5x/4] \times [y,5y/4]} Z(a,b) \geq - C \log n\right) \geq c\,.
    \end{equation*}
uniformly in $x,y \leq n/2.$  If for a set $S \subset [n] \times [n]$ we write $Z(S) = \sum_{(i,j) \in S} (M_\pi(i,j) - M_\tau(i,j))$ then for $(a,b) \in [x,5x/4] \times [y,5y/4]$ we can break up $$  Z(a,b) = Z(x,y) + Z([x] \times [y,b]) + Z([x,a] \times [y]) + Z([x,a] \times [y,b])\,.$$

The idea is that typically the maximum of the last three terms is order $\sqrt{xy/n} + \log n$.  Conditioning on a typical event for these last three terms, we then may use a central limit theorem to show that $Z(x,y)$ is larger than any fixed multiple of $\sqrt{xy/n}$ with probability bounded below.  The central challenge is showing that the maximum of the last three terms is not too large.  The middle two terms are easily handled by Freedman's inequality for martingales (see \cref{lem:upper-bound-sides}). The last requires a more delicate approach.

To show $$\E \max_{(a,b)  \in [x,(5/4)x] \times [y,(5/4)y]} |Z([x,a] \times [y,b])| \leq C\left(\sqrt{x y / n} + \log n\right)$$
we use a \emph{chaining} approach, a technique used to bound the expected maximum of a stochastic process.  We now provide a quick description of how chaining works in our context; for a more extensive background on chaining, see \cite[Chapter 8]{vershynin2018high} and the references therein.  

For simplicity of notation, suppose we are interested in bounding $\E \max_{(a,b)  \in [x] \times [y]} |Z(a,b)|\,. $
The strategy is to break up the set of $(a,b) \in [x] \times [y]$ into dyadic scales, e.g.\ the $k$th scale consists of all $(i\cdot x/2^k, j \cdot y/2^k)$ for $i \leq 2^k, j \leq 2^k$.  The idea is to approximate each point $(a,b)$ at each dyadic scale and note that the difference between the approximations at scales $k$ and $k+1$ is a random variable of variance $O(xy/(n2^{k}))$.  Further, the number of points at  scale $k$  is increasing exponentially.  The exponential decrease in variance leads to a strong enough tail bound (due to Freedman's inequality) that one can union bound over all points at scale $k$ and $k+1$ to uniformly control this error.

We note two subtle points: one reason why chaining works in this case is implicitly because the set of rectangles has \emph{bounded VC (Vapnik–Chervonenkis) dimension}, which is an assumption on the complexity on the set of indices that allows one to construct efficient nets; in this case, our efficient nets are simply the dyadic rectangles.  Chaining in general works well on classes of sets of bounded VC dimension and we refer the reader to \cite[Chapter 8.3]{vershynin2018high} for more context.  The other subtle point is that since we are dealing with sums of sparse random variables, the tail behavior of $Z(a,b)$ is Poisson-like rather than sub-gaussian: in some regimes the tail is gaussian-like and in others it is only exponential.  This complicates the chaining picture slightly.  Essentially, we only chain down to the scale at which the exponential tail overtakes the sub-gaussian tail; at this point, we use the exponential tail to obtain a logarithmic error, which we are able to tolerate due to our $\Omega(\log n)$ wiggle room.  We handle the chaining argument in \cref{sec:chaining}.

As a final note, we speculate on the asymptotic behavior of $\P(\pi \leq \tau).$  For the Brownian sheet $B(s,t)$ it was proven by Molchan \cite{molchan2022persistence} that there is some $\psi > 0$ so that \begin{equation} \label{eq:brownian-sheet-persistence}
    \P\left(\min_{s,t \leq T} B(s,t) \geq -1 \right) = \exp\left( -(\psi + o(1))\log^2 T\right)\,.
\end{equation}

We conjecture that in fact \begin{equation}\label{eq:conjecture}
    \P(\pi \leq \tau) = \exp\left(- (\psi + o(1) \right)\log^2 n)
\end{equation}
for the same $\psi$ as in \eqref{eq:brownian-sheet-persistence}.  As a justification for \eqref{eq:conjecture}, we first suggest heuristically that \begin{equation*}
    \P(\pi \leq \tau) = \exp(-o(\log^2n)) \P( Z(a,b) \geq 0 \text{ for all }a,b \in [n^{1 - o(1)}])^4\,.
\end{equation*}
Since $Z(n^{1 - o(1)},n^{1 - o(1)})$ has variance on the order of $n^{1 + o(1)}$,  one expects that \begin{equation} \label{eq:conjecture-strong-approx}
    (Z(a,b))_{(a,b) \in [n^{1 + o(1)}] \times  [n^{1 + o(1)}]} \approx (B(s,t))_{(s,t) \in [n^{1/2 + o(1)}] \times [n^{1/2 + o(1)}]}\,.
\end{equation}

In fact, our proof of the upper bound of \cref{th:main} can be understood as showing that the approximate distributional identity in \eqref{eq:conjecture-strong-approx} holds up to the smaller scale of $(a,b) \in [n^{7/12}] \times [n^{7/12}]$.
Since $\log^2 (n^{1/2 + o(1)}) = (1/4 + o(1))\log^2 n$, this provides a heuristic justification for our conjecture \eqref{eq:conjecture}.

\section{Preliminaries}
Let $\pi, \tau \in S_n$ be independently and uniformly chosen at random.  We will often write $\P_n$ in order to denote the dependence on $n$.  Let $M_\pi$ and $M_\tau$ denote the permutation matrices corresponding to $\pi$ and $\tau$.
For a set $A \subset [n] \times [n]$ let $X(A) = \sum_{(i,j) \in A} M_\pi(i,j)$, $Y(A) = \sum_{(i,j)\in A} M_\tau(i,j)$ and $Z(A) = X(A) - Y(A)$.  We will also write $X(i,j) = X([i] \times [j])$ and define $Y(i,j)$ and $Z(i,j)$ similarly.  
Then \cref{th:main} is equivalent to showing \begin{equation*}
    \P_n\left( Z(i,j) \geq 0 \text{ for all } i,j \in [n] \right) = \exp\left(-\Theta(\log^2 n) \right)\,.
\end{equation*}

For a set $A \subset [n] \times [n]$ we have $\E X(A) = |A|/n$ and so we define $\widetilde{X}(A) = X(A) - |A|/n$ and define $\widetilde{Y}(A)$ analogously.   As a result, for each fixed $A$ we have $\E \widetilde{X}(A) = 0$.   

For each rectangle $A$, we will see that $X(A)$ is a \emph{hypergeometric} random variable.  For parameters $N$ and $A, B \leq N$, a hypergeometric random variable $\xi \sim \hypergeom(N,B,A)$ may be defined combinatorially as follows: given $N$ total objects with $A$ red objects and $N-A$ blue objects, pick $B$ objects without replacement and let $\xi$ be the number of red objects picked.  For $\xi \sim \hypergeom(N,B,A)$ we recall \begin{equation}\label{eq:hypergeom-mean-variance}
    \E \xi = \frac{AB}{N}, \qquad \Var(\xi) = \frac{AB(N-A)(N - B)}{N^2 (N - 1)}\,.
\end{equation}

We first claim that for any box $B$, $X(B)$ is a hypergeometric random variable.  Throughout, for a permutation matrix $P$ and a set $S \subset [n] \times [n]$ we write $P_{S}$ to be the portion of the matrix indexed by $S$ and $|P_S|$ to be the number of ones in $P_S$.

\begin{lemma}\label{lem:count=hypergeom}
	Let $P$ be a random permutation matrix. For any box $B = (a_1,a_2] \times (b_1,b_2]$ we have that $$|P_B| \sim \hypergeom(n,b_2 - b_1, a_2 - a_1)\,.$$
\end{lemma}
\begin{proof}
    Consider the $n$ possible columns given by $n$ coordinate vectors.  The set of columns giving $B$ consists of $b_2 - b_1$ columns sampled uniformly without replacement from these $n$ total.  A column contains a $1$ if it is from the $a_2 - a_1$ columns containing a $1$ in a row in $(a_1,a_2]$.  This precisely matches the definition of the hypergeometric distribution.
\end{proof}

We will require tail bounds for hypergeometric random variables as well as a central limit theorem.  We will ultimately deduce a Bernstein-like tail bound from Freedman's inequality.  We recall a version of Freedman's inequality that not only bounds a martingale but also bounds the running maximum of a martingale \cite[(1.6)]{freedman1975tail}:

\begin{theorem}[Freedman's inequality] \label{th:freedman}
    Let $(X_k,\mathcal{F}_k)_{k \geq0}$ be a martingale with increments $\xi_k = X_k - X_{k-1}$.  Suppose that $|\xi_k| \leq M$ almost surely.  Define the quadratic variation $$V_n = \sum_{k = 1}^n \E[ \xi_k^2 \,|\,\mathcal{F}_{k-1}]\,.$$ Then for all $t, s > 0$ we have \begin{equation*}
        \P(\exists~k \leq n : |X_k - X_0| \geq t \text{ and } V_k \leq s^2) \leq 2\exp\left(-\frac{t^2}{2(s^2 + Mt/3)} \right)\,.
    \end{equation*}
\end{theorem}
From Freedman's inequality, a bound on hypergeometric random variables follows quickly:

\begin{corollary}\label{th:Bernstein}
    Let $a\leq b \leq 3n/4$ and set $X \sim \hypergeom(n,b,a).$  Then \begin{align*}
        \P(|X - \E X| \geq t) \leq 2 \exp\left(- \frac{1}{16} \min\left\{ \frac{t^2}{ab/n}, t\right\}\right)\,.
    \end{align*}
\end{corollary}
\begin{proof}
    Recall that we may sample $X$ by drawing $b$ objects total from a collection of $n$ objects consisting of $a$ red objects and $n - a$ blue objects.  Set $\mathcal{F}_j$ to be the $\sigma$-field generated by the first $j$ draws and define $X_k = \E[X\,|\,\mathcal{F}_k]$.  Note that $X_{k} - X_{k-1}$ is given by a centered Bernoulli variable, whose success parameter $p_k$ is at most $a/(n-b) \leq \frac{4a}{n}.$  This shows $|X_{k} - X_{k-1}| \leq 1$ and that $V_b \leq \frac{4ab}{n}$ almost surely.  Applying \cref{th:freedman} shows \begin{equation*}
        \P(|X - \E X| \geq t) = \P(|X_b - X_0| \geq t) \leq 2\exp\left(- \frac{t^2}{2(4ab/n + t/3)} \right) \leq 2 \exp\left(- \frac{1}{16} \min\left\{ \frac{t^2}{ab/n}, t\right\}\right)\,. \qedhere
    \end{equation*}
\end{proof}

We require a central limit theorem for hypergeometric random variables.  We isolate the following statement, which follows from e.g.\ \cite[Thm.~2.2]{lahiri2007berry}

\begin{theorem}\label{th:CLT}
    For each $M \geq 0$ there are constants $c,C > 0$ so that the following holds.  Let $W \sim \hypergeom(n,a,b)$ and set $\sigma^2 = \Var(W)$.  If $\Var(W) \geq C$ then $$\P(W - \E W \geq M \sigma) \geq c\,.$$ 
\end{theorem}

\section{Upper Bound}\label{sec:UB}

We prove the upper bound of \cref{th:main}.  
\begin{proposition}\label{pr:upper-bound}
    There is a universal constant $c > 0$ so that for $n \geq 2$ the following holds.  Let $\pi, \tau \in S_n$ be independently and uniformly chosen at random.  Then $$\P_n(\pi \leq \tau) \leq \exp(-c (\log n)^2)\,.$$
\end{proposition}

We first compare the permutation matrices underlying $Z$ with independent matrices with i.i.d.\ $\mathrm{Bernoulli}(1/n)$ entries.  We show that on an appropriately chosen submatrix, the likelihood ratio is essentially $1$.
\begin{lemma} \label{lem:Bernoulli-comparison}
    Set $N = \lfloor n^{7/12} \rfloor$ and let $M$ be an $N \times N$ matrix with entries in $\{0,1\}$.  Let $P$ be a random $n \times n$ permutation matrix and $Q$ a random $N \times N$ matrix with i.i.d.\ $\mathrm{Bernoulli}(1/n)$ entries.  Then uniformly among all $M$ with $\sum_{i,j} M(i,j) \leq n^{1/5}$ so that there is at most a single $1$ in each row and column we have
    \begin{equation*}
        \P(P_{[N] \times [N]} = M) = (1 + o(1))\P(Q = M)\,.
    \end{equation*} 
\end{lemma}

We prove \cref{lem:Bernoulli-comparison} in \cref{sec:Bernoulli-comparison}.  Our next step is to use a theorem of Rio in order to couple sums over rectangles in i.i.d.\ matrices with subexponential entries with a corresponding sum with Gaussian entries.  The following theorem is a consequence of \cite[Theorem 2.4]{rio1995strong}:\footnote{Rio's work is stated in terms of a bound on the number of moments and the distribution function.  Under our assumption of an exponential moment, one may uniformly bound the distribution function with Markov's inequality.}
\begin{theorem}\label{th:KMT}
    Let $\xi$ be a random variable with $\E \xi = 0, \E \xi^2 = 1$ and $\E e^{t \xi}\leq 2$ for all $|t| \leq c_0$.  Let $\xi_{i,j}$ be an array of i.i.d.\ copies of $\xi$.  Then there is a constant $C > 0$ depending only on $c_0$ so that we may couple $\xi_{i,j}$ with an array of i.i.d.\ standard Gaussian random variables $g_{i,j}$ so that \begin{equation*}
        \P\left(\max_{a,b \leq n} \left|\sum_{i \leq a, j \leq b} (\xi_{i,j} - g_{i,j}) \right| \geq C \log^2 n\right) \leq 2 e^{-\log^2 n}\,. 
    \end{equation*}
\end{theorem}

Finally, we will require an upper bound on the Gaussian persistence problem:
\begin{lemma}\label{lem:gaussian-persistence}
    For each $C > 0$ there is a constant $c > 0$ so that \begin{equation*}
        \P\left(\min_{a,b \leq n} \sum_{i \leq a, j \leq b} g_{i,j} \geq -C \log^2 n \right) \leq 2\exp\left(-c \log^2 n\right)\,.
    \end{equation*}
\end{lemma}

We show that \cref{lem:gaussian-persistence} follows from a general theorem by Li and Shao \cite{li2004lower} for Gaussian processes in \cref{sec:brownian-sheet}.  We now deduce the upper bound \cref{pr:upper-bound}.

\begin{proof}[Proof of \cref{pr:upper-bound}]
    We first perform our Bernoulli replacement.  Set $N = \lfloor n^{7/12} \rfloor$ and define \begin{equation*}
        \mathcal{T} = \{ X(N,N) \leq n^{1/5} , Y(N,N) \leq n^{1/5} \}\,.
    \end{equation*}
    \cref{th:Bernstein} shows \begin{equation*}
        \P(\mathcal{T}^c) \leq \exp(- n^{\Omega(1)}) \leq \exp(-\log^2 n)\,.
    \end{equation*}
    We then bound: \begin{equation}\label{eq:truncate-pre-bernoulli}
         \P_n(\pi \leq \tau) =\P_n\left(\min_{a \leq n, b\leq n} Z(a,b) \geq 0\right) \leq \P_n\left(\min_{a\leq N, b \leq N} Z(a,b) \geq 0  \cap \mathcal{T}\right) + e^{-\log^2 n} \,.
    \end{equation}
    Consider the random variables $\zeta_{i,j}$ with common distribution $\zeta$ given by \begin{equation*}
        \P(\zeta = 1) = \P(\zeta = -1) = \frac{1}{n}\left(1 - \frac{1}{n}\right) \qquad \text{ and }\qquad\P(\zeta = 0) = 1 - \frac{2}{n}\left(1 - \frac{1}{n}\right)
    \end{equation*}
    and note that $\zeta$ is the difference of two independent copies of $\mathrm{Bernoulli}(1/n)$ random variables.  By \cref{lem:Bernoulli-comparison}, for $n$ sufficiently large we have \begin{align}\label{eq:Z-by-zeta}
        \P_n\left(\min_{a\leq N, b \leq N} Z(a,b) \geq 0  \cap \mathcal{T}\right) \leq 2\P\left( \min_{a \leq N, b \leq N} \sum_{i \in [a], j \in [b]} \zeta_{i,j} \geq 0 \right)\,.
    \end{align}
    We now group the random variables into blocks with variance $1 + o(1)$.  For this, define $R = \lfloor \sqrt{n} \rfloor$ and $M = \lfloor N / R\rfloor = \Theta(n^{1/12})$.  For $i,j \in \{0,1,\ldots,M-1\}$, define the random variables $\xi_{i,j}$ and $\overline{\xi}_{i,j}$ via \begin{align*}
        \xi_{i,j} = \sum_{k \in [iR,(i+1)R)), \ell \in [jR,(j+1)R)} \zeta_{k,\ell},\qquad \overline{\xi}_{i,j} = \frac{\xi_{{i,j}}}{\sqrt{\Var(\xi_{i,j})}}
    \end{align*}
    and bound  \begin{equation}\label{eq:renormalization}
        \P\left( \min_{a \leq N, b \leq N} \sum_{i \in [a], j \in [b]} \zeta_{i,j} \geq 0  \right) \leq \P\left( \min_{a \leq M, b \leq M} \sum_{i \in [a], j \in [b]} \overline{\xi}_{i,j} \geq 0 \right)\,.
    \end{equation}
    Seeking to apply \cref{th:KMT}, note that $\E \overline{\xi}_{i,j} = 0$ and $\E \overline{\xi}_{i,j}^2 = 1$.  Set $\alpha = 1/n(1 - 1/n)$ and bound \begin{equation*}
        \E \exp(t \xi_{i,j}) = \left(2\alpha \cosh(t) + 1 - 2\alpha \right)^{R^2} \leq \exp\left(2\alpha R^2 (\cosh(t) - 1) \right) = \exp((2 + o(1)) (\cosh(t) - 1))
    \end{equation*}
    which can be made $\leq 2$ by requiring, say, $|t| \leq 1/2$.  \cref{th:KMT}  and \cref{lem:gaussian-persistence} show \begin{equation}\label{eq:replace-gaussian}
        \P\left( \min_{a \leq M, b \leq M} \sum_{i \in [a], j \in [b]} \overline{\xi}_{i,j} \geq 0 \right) \leq e^{-\Omega(\log^2 n)} +\P\left( \min_{a \leq M, b \leq M} \sum_{i \in [a], j \in [b]} g_{i,j} \geq -C\log^2 n \right) \leq e^{-\Omega(\log^2 n)} \,.
    \end{equation} 
    Combining equations \eqref{eq:truncate-pre-bernoulli}, \eqref{eq:Z-by-zeta}, \eqref{eq:renormalization} and \eqref{eq:replace-gaussian} completes the proof.
\end{proof}

\subsection{Proof of \cref{lem:Bernoulli-comparison}}\label{sec:Bernoulli-comparison}

The main step is to show that the random matrices $P_{[N] \times [N]}$ and $Q$ have asymptotically the same distribution for the number of $1$'s.  We recall that by \cref{lem:count=hypergeom} the number of $1$'s in $P_{[N] \times [N]}$ is a hypergeometric random variable.

\begin{lemma}\label{lem:count-bernoulli-perm}
    Set $N = \lfloor n^{7/12} \rfloor$ and let $X \sim \hypergeom(n,N,N)$ and $Y \sim \mathrm{Binomial}(N^2,1/n)$.  Then uniformly for $k \leq n^{1/5}$ we have \begin{align*}
        \frac{\P\left(X = k\right)}{\P\left(Y = k\right)} = 1 + o(1)
    \end{align*}
    for $n$ sufficiently large.
\end{lemma}
\begin{proof}
    First compute \begin{align*}
        \P(Y = k) &= \binom{N^2}{k}n^{-k}(1 - n^{-1})^{N^2 - k} = (1 + O(k^2 / N^2)) \frac{(N^{2} / n)^k}{k!}(1 + O(N^2/n^2))e^{- N^2 / n} \\
        &=(1 + o(1))\frac{(N^{2} / n)^k}{k!} e^{- N^2 / n}  \,.
    \end{align*}
    We then compute \begin{align*}
        \P(X = k) = \frac{\binom{N}{k} \binom{n - N}{N - k}}{\binom{n}{N}} = \frac{(N)_k^2}{k!} \cdot \frac{(n- N)_{N-k}}{(n)_N}\,.
    \end{align*}
    First note that $(N)_k^2 = N^{2k}(1 + O(k^2 / N)) = N^{2k}(1 + o(1)).$  Second, note \begin{align*}
        \frac{(n- N)_{N-k}}{(n)_N} &= (1 + o(1))n^{-k} \cdot \frac{(n - N)_N}{(n)_N} =(1 + o(1))n^{-k} \cdot \prod_{j = 0}^{N-1} \frac{n - N -j}{n - j} = (1 + o(1)) n^{-k} (1 - N/n)^N \\
        &= (1 + o(1))  n^{-k} e^{-N^2/n}
    \end{align*}
    where in the last asymptotic equality we used that $N^3/n^2 = o(1)$.  
\end{proof}

We now need to see that typically we do not have any rows or columns in $Q$ with two or more $1$'s.  Set $\mathcal{Q} = \{\exists~i : \sum_{j} Q(i,j) \geq 2 \text{ or }\sum_j Q(j,i) \geq 2  \}\,.$

\begin{lemma}\label{lem:cQ-likely}
    Let $N = \lfloor n^{7/12} \rfloor$, and $Q$ a random $N \times N$ matrix with i.i.d.\ $\mathrm{Bernoulli}(1/n)$ entries.  Let $|Q| = \sum_{i,j} Q(i,j)$ be the total number of $1$'s.  Then uniformly in $k \leq n^{1/5}$ for $n$ sufficiently large we have \begin{equation*}
        \P( \mathcal{Q} \,|\, |Q| = k) \geq 1 - n^{-1/6}\,.
    \end{equation*} 
\end{lemma}
\begin{proof}
    By exchangeability of the $N^2$ entries, note that conditioned on $|Q| = k$, $Q$ is uniform on all $0$-$1$ matrices with exactly $k$ many $1$'s.  We may sample such a matrix by placing $k$ many $1$'s one-at-a-time without ever placing two in the same entry.  For each pair of $1$'s, the probability we place them in the same row or column is $O(1/N).$  By union bounding over all pairs, we see that \begin{equation*} 
    \P( \mathcal{Q}^c \,|\, |Q| = k) \leq O(k^2/N) = O(n^{-11/60})\leq n^{-1/6}\,. \qedhere \end{equation*}
\end{proof}

\begin{proof}[Proof of \cref{lem:Bernoulli-comparison}]
    Set $k = \sum_{i,j} M(i,j)$. By exchangeability we have that \begin{align}
    \P(P_{N \times N} = M) &= \P(|P_{N \times N}| = k) \P(P_{N \times N} = M \,|\,|P_{N \times N}| = k) \nonumber \\
    &= \P(|P_{N \times N}| = k) \P(Q = M \,|\,|Q| = k, \mathcal{Q})  \label{eq:bernoulli-proof-cond-1s}
    \end{align}
    where the last equality uses the fact that $(Q \,|\,|Q| = k, \mathcal{Q})$ and $(P_{N \times N} \,|\,|P_{N \times N}| = k)$ have the same distribution.  Write \begin{equation}\label{eq:expand-Q-prob}
        \P(Q = M \,|\,|Q| = k, \mathcal{Q}) = \frac{\P(Q = M)}{\P(|Q| = k, \mathcal{Q})} = \frac{\P(Q = M)}{\P(|Q| = k)\P(\mathcal{Q} \,|\, |Q| = k)} = (1 + o(1))\frac{\P(Q = M)}{\P(|P_{N \times N}| = k)}
    \end{equation}
    where in the last equality we used \cref{lem:count-bernoulli-perm} and \cref{lem:cQ-likely}.  Combining \eqref{eq:bernoulli-proof-cond-1s} and \eqref{eq:expand-Q-prob} completes the proof.
\end{proof}

\subsection{Proof of \cref{lem:gaussian-persistence}} \label{sec:brownian-sheet}

We use the following theorem of Li-Shao \cite[Theorem~2.2]{li2004lower}:
\begin{theorem} \label{th:general-gaussian-upper-bound}
    Let $(X_j)_{j \in [M]}$ be a mean-zero Gaussian vector so that $\Var(X_j) \geq x^2 $ for all $j$.  Suppose also that for each $i$ we have \begin{equation*}
        \sum_{j \in [M]} \frac{|\E[X_j X_i]  |}{\sqrt{\Var(X_j) \Var(X_i)}} \leq \frac{5}{4}\,.
    \end{equation*}
    Then \begin{equation*}
        \P( \min_{j \in [M]} X_j \geq -x) \leq e^{-M/10}\,.
    \end{equation*}
\end{theorem}

\begin{proof}[Proof of \cref{lem:gaussian-persistence}]
    Define $G(a,b) = \sum_{i \leq a, j\leq b} g_{i,j}$ and note that $(G(a,b))_{a,b \leq n}$ is a mean-zero Gaussian vector.  Compute \begin{align*}
        \E [G(a_1,b_1) G(a_2,b_2)] &= \E \sum_{\substack{i_1 \leq a_1, j_1 \leq b_1 \\ i_2 \leq a_2, j_2 \leq b_2}} g_{i_1,j_1} g_{i_2,j_2} = \sum_{\substack{i_1 \leq a_1, j_1 \leq b_1 \\ i_2 \leq a_2, j_2 \leq b_2}} \one\{i_1 = i_2, j_1 = j_2\} 
        = \min\{a_1, a_2\} \min\{b_1,b_2\}\,. 
    \end{align*}
    For $\rho = 400$ and each fixed $(i,j)$, bound \begin{align}
        \sum_{(a,b)} \frac{\E[G(\rho^i,\rho^j) G(\rho^a,\rho^b)]}{\sqrt{\Var(G(\rho^i,\rho^j))\Var( G(\rho^a, \rho^b))}} &= \sum_{(a,b) } \rho^{-i/2-j/2-a/2-b/2} \rho^{\min\{i,a\}\min\{j,b\}} = \sum_{(a,b) }\rho^{-|i-a|/2 - |j-b|/2}  \nonumber \\
        &\leq \left(\sum_{k \in \Z} \rho^{-|k|/2} \right)^2 = \left(\frac{1 + \rho^{-1/2}}{1 - \rho^{-1/2}} \right)^2 \leq \frac{5}{4} \label{eq:corr-bound}
    \end{align}
    where the last inequality is true for $\rho = 400.$  We now define the set $\mathcal{S}_0 = \{i : \rho^i \in [\log^3 n, n]\}$ and note that $|\mathcal{S}_0| = \Theta(\log n).$  For all $(i,j) \in \mathcal{S}_0$ we have that $\Var(G(\rho^i,\rho^j)) \geq \log^6 n \geq C^2 \log^4 n$ for $n$ sufficiently large.  By \cref{th:general-gaussian-upper-bound} we have \begin{align*}
        \P(\min_{a,b \leq n} G(a,b) \geq -C\log^2 n) \leq \P(\min_{i,j \in \mathcal{S}_0^2} G(\rho^i,\rho^j) \geq -C \log^2 n) \leq e^{-|\mathcal{S}_0|^2 / 10} = e^{-\Omega( \log^2 n)}
    \end{align*}
    where we used \eqref{eq:corr-bound} to verify the hypotheses of \cref{th:general-gaussian-upper-bound}.    
\end{proof}

\section{Lower bound} \label{sec:LB}

The goal is to prove the lower bound of \cref{th:main}.  
\begin{proposition}\label{prop:LB}
    There is a universal constant $C > 0$ so that for $n \geq 2$ the following holds.  Let $\pi, \tau \in S_n$ be independently and uniformly chosen at random. Then $$\P_n(\pi \leq \tau) \geq \exp\left(- C(\log n)^{2} \right)\,.$$
\end{proposition}

We prove \cref{prop:LB} in two main steps.  First, we will use a correlation inequality to show to reduce to the product of probabilities over various boxes.  Define $x_i = (5/4)^i$ and $y_j = (5/4)^j$.  Define $k_0 = \max\{i : x_i \leq n/2\}$.  We will first show the following lower bound using a correlation inequality together with conditioning on the first and last $O(\log n)$ elements of the permutations.
\begin{lemma} \label{lem:LB-setup}
    For each $C >0$ there is a $C' > 0$ so that for $n$ sufficiently large we have
    \begin{equation*}
        \P_n(\pi \leq \tau) \geq e^{-C'(\log n)^2} \prod_{i,j \leq k_0} \P_{n- 2 C\log n}\left(\min_{(a,b) \in [x_i,5x_i/4] \times [y_j,5y_j/4]} Z(a,b) \geq - C \log n\right)^4 \,.
    \end{equation*}    
\end{lemma}

We will prove \cref{lem:LB-setup} in \cref{sec:LB-setup} using a correlation inequality of Johnson-Leader-Long for random permutations in the strong Bruhat order.  We then uniformly bound the probabilities in \cref{lem:LB-setup}.

\begin{lemma}\label{lem:lower-bound-one-box}
    There are constants $C,c>0$ so that for all $x,y \leq n/2$ we have \begin{equation*}
        \P\left(\min_{(a,b) \in [x,5x/4] \times [y,5y/4]} Z(a,b) \geq - C \log n\right) \geq c\,.
    \end{equation*}
\end{lemma}

The proof of \cref{prop:LB} now follows quickly:

\begin{proof}[Proof of \cref{prop:LB}]
    We note that $k_0 = \Theta(\log n)$, and so combining \cref{lem:LB-setup} with \cref{lem:lower-bound-one-box} completes the proof.
\end{proof}

\subsection{Proof of \cref{lem:LB-setup}} \label{sec:LB-setup}

We note that $S_n$ under the strong Bruhat order is not a lattice and so does not satisfy the usual assumptions of the Fortuin–Kasteleyn–Ginibre (FKG) inequality.  However, a theorem of Johnson, Leader and Long \cite{johnson2020correlation} proves that the corresponding FKG inequality holds for the strong Bruhat order. Recall that $A \subset S_n$ is \emph{increasing} if whenever $\pi_1 \in A$ and $\pi_2 \geq \pi_1$ then $\pi_2 \in A$.

\begin{theorem}[Johnson-Leader-Long, \cite{johnson2020correlation}]\label{thm:JLL-FKG}
    Let $A, B \subset S_n$ be increasing events in the strong Bruhat order and let $\P$ denote the uniform probability on $S_n$.  Then $\P(A \cap B) \geq \P(A) \P(B).$
\end{theorem}

We quickly note that for events $A,B$ we have \begin{align*}
    \P(A^c \cap B^c) - \P(A^c) \P(B^c) &= 1 - (\P(A) + \P(B) - \P(A\cap B)) - (1 - \P(A))(1 - \P(B)) \\
    &= \P(A \cap B) - \P(A) \P(B) 
\end{align*}
and so \cref{thm:JLL-FKG} holds for decreasing events as well.

\begin{corollary}\label{cor:FKG}
Let $A,B \subset S_n \times S_n$ be increasing in the first coordinate and decreasing in the second coordinate, and let $\P$ denote the uniform probability on $S_n \times S_n$.  Then $\P(A \cap B) \geq \P(A)\P(B).$
\end{corollary}
\begin{proof}
    Let $(\pi,\tau) \in S_n \times S_n$ be chosen uniformly at random.  Set $A_\pi = \{\tau : (\pi,\tau) \in A\}$ and define $B_\pi$ similarly.  Note that $A_\pi, B_\pi$ are decreasing events by assumption.  We can then write \begin{align*}
        \P(A \cap B) = \E_\pi\left[ \P_\tau(A_\pi \cap B_\pi )  \right] \geq \E_\pi[\P_\tau(A_\pi) \P_\tau(B_\pi)]
    \end{align*}
    where in the last inequality we used \cref{thm:JLL-FKG} for the decreasing events $A_\pi, B_\pi$.  We now may write \begin{align*}
        \E_\pi[\P_\tau(A_\pi) \P_\tau(B_\pi)] &= \int_0^\infty \int_0^\infty \P_\pi(\P_\tau(A_\pi) > s, \P_\tau(B_\pi) > t)\,ds\,dt \\
        &\geq \int_0^\infty \int_0^\infty \P_\pi(\P_\tau(A_\pi) > s)\P_\pi(\P_\tau(B_\pi) > t)\,ds\,dt \\
        &= \E_\pi[\P_\tau(A_\pi)] \cdot \E_{\pi}[\P_\tau(B_\pi)] \\
        &= \P(A) \P(B)
    \end{align*}
    where in the inequality we used that $A$ and $B$ are increasing in $\pi$ and applied \cref{thm:JLL-FKG}.
\end{proof}

For $N = \lceil n/2\rceil$ define the events \begin{gather*}
    \cE_1 = \left\{\min_{a,b \in[ N]\times[N]} Z(a,b) \geq 0\}\right\}\,, \quad \cE_2 = \left\{\min_{a,b \in[ n - N,n]\times[N]} Z(a,b) \geq 0\}\right\} \\
    \cE_3 = \left\{\min_{a,b \in[ N]\times[n-N,n]} Z(a,b) \geq 0\}\right\}, \quad \cE_4 = \left\{\min_{a,b \in[n - N,n]\times[n-N,n]} Z(a,b) \geq 0\}\right\}\,.
\end{gather*}

Since $\{\pi \leq \tau\} = \cE_1 \cap \cE_2 \cap \cE_3 \cap \cE_4$ we have \begin{equation}\label{eq:LB-four-events}
    \P_n(\pi \leq \tau) = \P(\cE_1 \cap \cE_2 \cap \cE_3 \cap \cE_4) \geq \prod_{j = 1}^4 \P(\cE_j)
\end{equation}
where in the last inequality we used \cref{cor:FKG}.  
We now show that the events $\cE_j$ have the same probability. \begin{lemma}\label{lem:E-equiprobable}
    We have $\P(\cE_1) = \P(\cE_2) = \P(\cE_3) = \P(\cE_4)$.
\end{lemma}
\begin{proof}
    For all $a,b$ we have that $X([a] \times [b]) = b - X((a,n] \times [b])$. This means that \begin{align*}
        \min_{a,b \in[ n - N,n]\times[N]} Z(a,b) = \min_{a,b \in[ n - N,n]\times[N]}  (-Z((n-a,n] \times [b]))\,.
    \end{align*}
    However, by replacing the first $a$ rows with the last $a$ rows, we have that \begin{equation*}
        \min_{a,b \in[ n - N,n]\times[N]}  (-Z((n-a,n] \times [b])) \overset{d}{=} \min_{a' \times b \in [N]\times[N]}  (-Z(a',b)) \overset{d}{=} \min_{a \times b \in [N]\times[N]}  Z(a',b) 
    \end{equation*}
    where in the second equality we used that $Z$ and $-Z$ have the same distribution.  This shows $\P(\cE_1) = \P(\cE_2)$.  By taking the transpose we see $\P(\cE_2) = \P(\cE_3)$.  Replacing the first $N$ rows with the last $N$ rows and similarly for the columns shows $\P(\cE_1) = \P(\cE_4).$
\end{proof}

We are now ready to prove \cref{lem:LB-setup}: 
\begin{proof}[Proof of \cref{lem:LB-setup}]
    By \eqref{eq:LB-four-events} and \cref{lem:E-equiprobable} it is sufficient to show that \begin{equation*}
        \P(\cE_1) \geq  e^{-C'(\log n)^2} \prod_{i,j \leq k_0} \P_{n - 2C\log n}\left(\min_{(a,b) \in [x_i,5x_i/4] \times [y_j,5y_j/4]} Z(a,b) \geq - C \log n\right)\,.
    \end{equation*}
    Recall that $Z(a,b) = X(a,b) - Y(a,b)$.  Let $\pi$ denote the permutation corresponding to $X$ and $\tau$ the permutation corresponding to $Y$.  Define the events \begin{align*}
        \cG' := \{\pi(j) = j\ \forall\ j \in [C \log n] \cup [n - C\log n,n]\} \cap \{\tau(j) = n - j+1 \ \forall\ j \in [C \log n] \cup [n - C\log n,n]\}
    \end{align*}
    and compute $\P(\cG') = n^{-(4 + o(1))C\log n} = \exp(-(4C + o(1))\log^2n)$.   If we condition on $\cG'$,the remaining portion of the permutations $\pi$ and $\tau$ are uniform at random permutations on $n - 2 C \log n$ many elements.  Let $\pi'$ and $\tau'$ denote these permutations.  
    For $i,j \leq k_0$ define \begin{align*}
        \cG_{i,j} = \left\{\min_{(a,b) \in [x_i,5x_i/4] \times [y_j,5y_j/4]} Z_{\pi',\tau'}(a,b) \geq - C \log n \right\}
    \end{align*}
    and note that $Z_{\pi',\tau'}(a,b) = Z_{\pi,\tau}([C\log n , C\log n + a] \times [C\log n, C\log n + b]) = Z_{\pi,\tau}(C\log n + a , C\log n + b) -  C\log n$.  Thus we have \begin{align*}
        \P(\cE_1) &\geq \P(\cG '\cap \bigcap_{i,j \leq k_0} \cG_{i,j}) = e^{-(4 + o(1)C) \log^2 n} \P( \bigcap_{i,j \leq k_0} \cG_{i,j} \,|\,\cG') \\
        &\geq e^{-C'(\log n)^2} \prod_{i,j \leq k_0} \P_{n - 2C\log n}\left(\min_{(a,b) \in [x_i,5x_i/4] \times [y_j,5y_j/4]} Z(a,b) \geq - C \log n\right)
    \end{align*}
    where in the last inequality we used \cref{cor:FKG}.
\end{proof}

\subsection{Proof of \cref{lem:lower-bound-one-box}}

In order to prove \cref{lem:lower-bound-one-box}, we will break the box $[a] \times [b]$ into four sections, and write
\begin{align}
    Z(a,b) &= Z(x,y) + Z([x] \times [y,b]) + Z([x,a] \times [y]) + Z([x,a] \times [y,b]) \nonumber \\
    &\geq Z(x,y) - \max_{b \in [y,(5/4)y]} |Z([x] \times [y,b])| - \max_{a \in [x,(5/4)x]} |Z([x,a] \times [y])| \nonumber \\
    &\qquad - \max_{(a,b) \in [x,(5/4)x] \times[y,(5/4)y]} |Z([x,a] \times [y,b])|  \label{eq:rectangle-breakdown} \,. 
\end{align}

We will first prove that in expectation, the last three terms are of order $\sqrt{xy/n} + \log n$.  We then will show that with constant probability $Z(x,y)$ is positive enough to make up for it.  We handle the middle two terms in \eqref{eq:rectangle-breakdown} with Freedman's inequality:

\begin{lemma}\label{lem:upper-bound-sides}
    There is a universal constant $C > 0$ so that for all $x,y \leq n/2$ we have $$\E \max_{b \in[y,(5/4) y] }|\Xt([x] \times [y,b])| \leq C\left(\sqrt{\frac{xy}{n}} + 1\right) \,.  $$
\end{lemma}
\begin{proof}
    Define $M_j = \Xt([x] \times [y,y+j])$.  Note that $M_j$ is a martingale with bounded increments.  Further, as in the proof of \cref{th:Bernstein}, we have that the variance of the increments is bounded by $C x/n$.  In particular, \cref{th:freedman} shows \begin{equation*}
        \P\left(\max_{b \in[y,(5/4) y] }|\Xt([x] \times [y,b])| \geq t \sqrt{\frac{xy}{n}} \right) \leq \exp\left(- c\min\{t^2, t \sqrt{xy/n}\}  \right)\,.
    \end{equation*}
    Integrating gives \begin{align*}
       \E \max_{b \in[y,(5/4) y] }|\Xt([x] \times [y,b])| &\leq  \sqrt{\frac{xy}{n}} + \sqrt{\frac{xy}{n}} \int_1^\infty \exp(- c\{ t^2, t \sqrt{xy/n}\})\,dt \leq C\left(\sqrt{\frac{xy}{n}} + 1\right)\,.  \qedhere
    \end{align*}
\end{proof}

Handling the last term in \eqref{eq:rectangle-breakdown} is significantly trickier.  We prove an upper bound on the expectation via a chaining approach:

\begin{lemma}\label{lem:upper-bound-max-chaining}
    There is a universal constant $C> 0$ so that for all $x,y  \leq n/2$ we have $$\E \max_{(a,b)  \in [x,(5/4)x] \times [y,(5/4)y]} |\Xt([x,a] \times [y,b])| \leq C\left(\sqrt{x y / n} + \log n\right) \,.$$ 
\end{lemma}

We prove \cref{lem:upper-bound-max-chaining} in \cref{sec:chaining}.  For a universal constant $C_1 > 0$ define the events \begin{align*}
    \cG_1 &:= \left\{ \max_{b \in[y,(5/4) y] }|Z([x] \times [y,b])| \leq C_1 \left(\sqrt{x y / n} + \log n\right) \right\} \\
    \cG_2 &:=  \left\{ \max_{a \in[x,(5/4) x] }|Z([x,a] \times [y])| \leq C_1 \left(\sqrt{x y / n} + \log n\right) \right\} \\
    \cG_3 &= \left\{\max_{(a,b)  \in [x,(5/4)x] \times [y,(5/4)y]} |Z([x,a] \times [y,b])| \leq C_1 \left(\sqrt{x y / n} + \log n\right)\right\}
\end{align*}

By Markov's inequality together with \cref{lem:upper-bound-sides} and \cref{lem:upper-bound-max-chaining}, we may choose $C_1$ large enough so that \begin{equation}
    \P(\cG_1 \cap \cG_2 \cap \cG_3) \geq 1/2\,.
\end{equation}
In order to show that $Z(x,y)$ can be large even conditioned on $\cG_1 \cap \cG_2 \cap \cG_3$, we identify the conditional distribution for the number of $1$'s in a box of a random permutation conditioned on a portion out of the box.

\begin{lemma}\label{lem:box-cond-frame}
    Let $P$ be a random permutation matrix.  For $x_1 \leq x_2$ and $y_1 \leq y_2$, we have \begin{gather*}
       \left( |P_{[x_1] \times [y_1]}| \,|\, P_{[x_2] \times [y_2] \setminus [x_1] \times [y_1]} \right) \sim \hypergeom\left( n - (x_2 - x_1) - (y_2 - y_1) + m_3, y_1 - m_2, x_1 - m_1 \right) \\
        \text{ where }\quad  m_1 = |P_{[x_1] \times (y_1,y_2]}|, m_2 = |P_{(x_1,x_2] \times [y_1]}|, m_3 = |P_{(x_1,x_2] \times (y_1,y_2]}|\,.
    \end{gather*}
\end{lemma}
\begin{proof}
    Let $\pi\in S_n$ be the permutation corresponding to $P$, meaning row $j$ has its $1$ in column $\pi(j)$.  We view $\pi$ as a map from the set of rows to the set of columns. Define
    $$
    S_1=[x_1],\qquad S_2=(x_1,x_2],\qquad T_1=[y_1],\qquad T_2=(y_1,y_2].
    $$
    Condition on the pattern of $P$ on $[x_2]\times[y_2]\setminus([x_1]\times[y_1])$. This reveals: which rows of $S_1$ map into $T_2$ (there are $m_1$ of them);  which columns of $T_1$ are already used by rows in $S_2$ (there are $m_2$ of them); and
    which rows of $S_2$ map into $T_2$ (there are $m_3$ of them).
    
    Let $S_0\subset S_1$ be the set of rows not mapped into $T_2$ and note that $|S_0|=x_1-m_1$.
    We then see that all rows in $S_0$ must map into the set of columns
    $$
     C := [n]\setminus\left(\pi(S_2)\cup T_2 \right)\,.
    $$
    Since $|\pi(S_2)|=|S_2|=x_2-x_1$ and $|\pi(S_2)\cap T_2|=m_3$, the principle of inclusion-exclusion gives 
    $$
    |C| = n-(x_2-x_1)-(y_2-y_1)+m_3.
    $$
    Among the available columns $C$, the ones lying in $T_1$ are exactly $C_T := T_1 \setminus \pi(S_2)$ and so $|C_T|=y_1-m_2$.
    
    By exchangeability, the image $\pi(S_0)$ is a uniformly random $(x_1 - m_1)$-subset of $C$, and $|P_{[x_1] \times [y]}| = |\pi(S_0) \cap C_T|$.  This shows  $$\left( |P_{[x_1] \times [y_1]}| \,|\, P_{[x_2] \times [y_2] \setminus [x_1] \times [y_1]} \right) \sim \hypergeom(|C|, |C_T|, |S_0|)$$
    which completes the proof.
\end{proof}

Set $\sigma^2 = ab/n$ and for a constant $C > 0$ define $\cG_4^X$ via \begin{align*}
    \cG_4^X &:= \left\{ \left|X([a] \times (b,5b/4]) - \frac{ab}{4n} \right| \leq C\sigma \right\} \cap  \left\{ \left||X({(a,5a/4] \times [b]}) - \frac{ab}{4n} \right| \leq C\sigma \right\}  \\
    &\qquad\cap \left\{ \left|X((a,5a/4] \times (b,5b/4]) - \frac{ab}{16n} \right| \leq C\sigma \right\}\,.
\end{align*}
Define $\cG_4^Y$ analogously for $Y$, and set $\cG_4 = \cG_4^X \cap \cG_4^Y$.  
By \cref{th:Bernstein} and \cref{lem:count=hypergeom} we may increase $C$ so that $\P(\cG_4^c) \leq \frac{1}{4}$ and so \begin{equation}
    \P(\cG_1 \cap \cG_2 \cap \cG_3 \cap \cG_4) \geq 1/4\,.
\end{equation}

\begin{lemma}
    For each $M,C > 0$ there is a $c > 0$ so that \begin{equation*}
        \P( Z(a,b) \geq M \sigma - C\,|\,\cG_1 \cap \cG_2 \cap \cG_3 \cap \cG_4) \geq c\,. 
    \end{equation*} 
\end{lemma}
\begin{proof}
    First note that if $\sigma$ is smaller than a constant $C_0$ then we have $\P(|Z(a,b)| \leq 2 C_0) \geq  1/2$ by Markov's inequality.  We may thus assume that $\sigma > C_0$ for some $C_0$ depending on $M$ to be chosen later.
    For $P = M_\pi$, conditioned on $P_{[5a/4] \times [5b/4] \setminus [a] \times [b]}$, \cref{lem:box-cond-frame} tells us that $|P_{[a] \times [b]}|$ is a hypergeometric random variable.  By \cref{th:CLT}, it is sufficient to show that for all possible patterns $P_{[(5/4)a] \times [(5/4)b] \setminus [a] \times [b]}$ in $\cG_4$ we have that the conditioned mean of $|P_{[a] \times [b]}|$ is at most $O(\sigma)$ from the unconditional mean of $P_{[a] \times [b]}$.  Using the notation of \cref{lem:box-cond-frame}, define \begin{equation*}
        m_1 = |P_{[a] \times (b,5b/4]}|, \quad  m_2 = |P_{(a,5a/4] \times [b]}|, \quad m_3 = |P_{(a,5a/4] \times (b,5b/4]}|\,.
    \end{equation*}
    Write \begin{align*}
        \widetilde{m}_1 = m_1 - \frac{ab}{4n}, \quad \widetilde{m}_2 = m_2 - \frac{ab}{4n}, \quad \widetilde{m}_3 = m_3 - \frac{ab}{16n}
    \end{align*}
    and note that on $\cG_4$ we have that $|\widetilde{m}_j| \leq C\sigma$ for all $j \in \{1,2,3\}.$  Set $A = a - \frac{ab}{4n}, B = b - \frac{ab}{4n}$ and $N = n - \frac{a}{4} - \frac{b}{4} + \frac{ab}{16n}$ and note first that $\frac{AB}{N} = \frac{ab}{n}.$  
    By  \cref{lem:box-cond-frame}, the distance from the conditional mean of $|P_{[a_1] \times [b_1]}|$ to the unconditional mean is \begin{align*}
        \left|\frac{(A - \widetilde{m}_1)(B - \widetilde{m}_2)}{N + \widetilde{m}_3} - \frac{AB}{N} \right| &\leq \left|\frac{AB}{N+\widetilde{m}_3} - \frac{AB}{N} \right| + \left|\frac{A \widetilde{m}_2}{N + \widetilde{m}_3}\right| + \left|\frac{B \widetilde{m}_1}{N + \widetilde{m}_3}\right| + \left|\frac{\widetilde{m}_1 \widetilde{m}_2}{N + \widetilde{m}_3}\right| \\ 
        &\leq 4C \sigma\,.
    \end{align*}
    Seeking to apply \cref{th:CLT}, recall that we have assumed $\sigma \geq C_0$; taking $C_0$ large enough so that we may apply \cref{th:CLT} shows that we have $$\P(\Xt(a,b) \geq 2 M \sigma \,|\,\cG_1 \cap \cG_2 \cap \cG_3 \cap \cG_4) \geq c\,.$$
    Arguing similarly for $\widetilde{Y}$ completes the proof.
\end{proof}

\subsection{Proof of \cref{lem:upper-bound-max-chaining}} \label{sec:chaining}
For simplicity, write $\Xt(a,b) = \Xt([a] \times [b]).$
We have the distributional equality $$\left(\Xt([x,a] \times [y,b]) \right)_{(a,b) \in[x,(5/4)x] \times [y,(5/4)y]} \overset{d}{=} \left(\Xt(a,b) \right)_{(a,b) \in [x/4] \times [y/4]}\,.$$

  By potentially increasing the constant $C$ in the statement, it is sufficient to show that for $x \leq y \leq n/4$ each given by a powers of $2$ that we have 
    
\begin{equation}\label{eq:chaining-new}
    \E \max_{a \leq x, b \leq y } |\Xt(a,b)| \leq C\left(\sqrt{x y / n} + \log n\right)\,.
\end{equation}
Write $x = 2^{N_x}$, $y = 2^{N_y}$.  For each $k \leq N_y$ define the set of points \begin{equation}
    \mathcal{S}_k = \left\{ i \cdot \frac{x}{2^{k \wedge N_x}},j \cdot \frac{y}{2^{k}} \right\}_{i \leq 2^{k \wedge N_x}, j \leq 2^k}\,.
\end{equation}
Define $K = N_y \wedge \max\{k : \frac{xy}{n 2^k} \geq 1\}$.  For $(a,b) \in [x]\times[y]$ and $k \leq K$ we may choose $(a_k,b_k) \in \mathcal{S}_k$ so that 
\begin{equation} \label{eq:a_k-choice}
    |a_k - a| \leq \frac{x}{2^{k\wedge N_x}}, \quad |b_k - b| \leq \frac{y}{2^k}    \,.
\end{equation}
Write \begin{align*}
    |\Xt(a,b)| \leq |\Xt(a_K,b_K)| + |\Xt(a,b) - \Xt(a_K,b_K)|
\end{align*}
which implies \begin{equation} \label{eq:chaining-split-scales}
    \E \max_{a \leq x, b \leq y } |\Xt(a,b)| \leq \E \max_{(a_K,b_K) \in \mathcal{S}_K} |\Xt(a_K,b_K)| + \E \max_{a \leq x, b \leq y} \min_{(a_K,b_K) \in \mathcal{S}_K} |\Xt(a,b) - \Xt(a_K,b_K)|\,.
\end{equation}

We begin with the second term.

\begin{lemma}\label{lem:chaining-sparse-part}
    In the notation above, there is a universal constant $C > 0$ so that for $n \geq 2$ we have $$\E \max_{a \leq x, b \leq y} \min_{(a_K,b_K) \in \mathcal{S}_K} |\Xt(a,b) - \Xt(a_K,b_K)| \leq C \log n\,.$$
\end{lemma}
\begin{proof}
    To begin with, note that for each $a \leq x, b \leq y$ there is some $(a_K,b_K) \in \mathcal{S}_K$ so that \begin{equation}\label{eq:a_K-aspects}
        |a_K - a| \leq \frac{x}{2^{K \wedge N_x}} \quad  \text{and} \quad |b_K - b| \leq \frac{y}{2^K}. 
    \end{equation} As such, if $(a,b)$ and $(a_K,b_K)$ satisfy \eqref{eq:a_K-aspects}, then we may write the symmetric difference of the rectangles $[a] \times[b]$ and $[a_K,b_K]$ as the union of at most two rectangles of area at most $\frac{xy}{2^K}$. Note that $\frac{xy}{2^K n} \leq 2$ by assumption.  By \cref{th:Bernstein}, for all $t \geq 2$ we may thus bound \begin{equation*}
        \P\left(|\Xt(a,b) - \Xt(a_K,b_K)| \geq t \right)  \leq 4 \exp\left( - \frac{t}{32} \right)\,.
    \end{equation*}
    Since there are at most $n^4$ choices for $(a,b), (a_K,b_K)$ we obtain \begin{equation*}
        \P\left( \max_{a \leq x, b \leq y} \min_{(a_K,b_K) \in \mathcal{S}_K} |\Xt(a,b) - \Xt(a_K,b_K)|  \geq t\right) \leq 4n^4 \exp\left(- \frac{t}{32}\right)\,.
    \end{equation*}
    Upper bounding the expectation by integrating the tail probability we see \begin{align*}
       \E \max_{a \leq x, b \leq y} \min_{(a_K,b_K) \in \mathcal{S}_K} |\Xt(a,b) - \Xt(a_K,b_K)| &\leq 128 \log n + \int_{128 \log n}^\infty 4 n^4 \exp\left(- \frac{t}{32}\right)\,dt \\
       &= O(\log n) \,.\qedhere 
    \end{align*}
\end{proof}

We now handle the first term in \eqref{eq:chaining-split-scales}.

\begin{lemma}\label{lem:chaining-gaussian-part}
    There is a universal constant $C > 0$ so that for $n \geq 2$ we have 
    $$\E \max_{(a_K,b_K) \in \mathcal{S}_K} |\Xt(a_K,b_K)|  \leq C \left(\sqrt{\frac{xy}{n}} + \log n\right) \,.$$
\end{lemma}
\begin{proof}
    Note first that if $\frac{xy}{n} \leq 2$ then the argument in \cref{lem:chaining-sparse-part} shows an upper bound of $\log n$.  We may thus assume without loss of generality that $\frac{xy}{n} \geq 2$\,.  Letting $(a_k,b_k)$ denoting the choice from \eqref{eq:a_k-choice} we may write the telescoping sum:
    \begin{equation*}
        \Xt(a_K,b_K) = \Xt(a_0,b_0) + \sum_{k = 1}^{K} (\Xt(a_k,b_k) - \Xt(a_{k-1},b_{k-1}))\,. 
    \end{equation*}
    and so \begin{equation}\label{eq:chaining-telescope}
        \E \max_{(a_K,b_K) \in \mathcal{S}_K}|\Xt(a_K,b_K)| \leq  \E |\Xt(a_0,b_0)| + \sum_{k = 1}^{K} \E \max_{(a_k,b_k), (a_{k-1},b_{k-1}) \in \eqref{eq:a_k-choice} } |\Xt(a_k,b_k) - \Xt(a_{k-1},b_{k-1})| 
    \end{equation}
    where in the maximum we write $(a_k,b_k), (a_{k-1},b_{k-1}) \in \eqref{eq:a_k-choice}$ to denote that $(a_k,b_k) \in \mathcal{S}_k, (a_{k-1},b_{k-1}) \in \mathcal{S}_{k-1}$ and both satisfy \eqref{eq:a_k-choice} for the same $(a,b)$.

    We will show that there is an absolute constant $C > 0$ so that 
    \begin{equation} \label{eq:k-hierarchy}
         \E \max_{(a_k,b_k), (a_{k-1},b_{k-1}) \in \eqref{eq:a_k-choice} } |\Xt(a_k,b_k) - \Xt(a_{k-1},b_{k-1})| \leq C 2^{-k/8} \sqrt{\frac{xy}{n}} 
    \end{equation}
    for all $k \in [0,K]$ (where we interpret $\mathcal{S}_{-1} = \emptyset$). For pairs $(a_k,b_k),(a_{k-1},b_{k-1}) \in \eqref{eq:a_k-choice}$ we have that $|a_k - a_{k-1}| \leq \frac{4x}{2^{k \wedge N_x}}$ and $|b_k - b_{k-1}| \leq \frac{4y}{2^{k}}.$  In particular, the symmetric difference of the rectangles $[a_k,b_k]$ and $[a_{k-1},b_{k-1}]$ may be written as the union of at most two rectangles of area at most $4\frac{xy}{2^k}$.  Thus for a universal constant $c > 0$ if we set $\sigma^2 = xy/n$
    \begin{align*}
        \P\left(|\Xt(a_k,b_k) - \Xt(a_{k-1},b_{k-1})| \geq \theta \sigma 2^{-k/8}  \right) &\leq \exp\left(-c \min\left\{\frac{\theta^2 \sigma^2 / 2^{k/4}}{\sigma^2/ 2^k}, \theta \sigma 2^{-k/8} \right\}\right)  \\
        &\leq \exp\left(- c \min\left\{ \theta^2 2^{3k/4}, \theta 2^{3k/8} \right\} \right) 
    \end{align*}
    where in the second line we used that $\sigma^2 2^{-k} \geq 1$ by the definition of $K$.  There are at most, say, $16^k$ such choices for $(a_k,b_k),(a_{k-1},b_{k-1})$ and so we have  \begin{align*}
        \P\left( \max_{(a_k,b_k), (a_{k-1},b_{k-1}) \in \eqref{eq:a_k-choice} } |\Xt(a_k,b_k) - \Xt(a_{k-1},b_{k-1})|  \geq \theta \sigma 2^{-k/8} \right) \leq 16^{k} \exp\left(- c \min\left\{ \theta^2 2^{3k/4}, \theta 2^{3k/8} \right\} \right) \,.
    \end{align*}
    We may then bound \begin{align*}
        \E  \max_{(a_k,b_k), (a_{k-1},b_{k-1}) \in \eqref{eq:a_k-choice} }& |\Xt(a_k,b_k) - \Xt(a_{k-1},b_{k-1})|  \\
        &\leq \sigma 2^{-k/8} + \sigma 2^{-k/8} \int_{1}^\infty 16^k \exp\left(- c \min\left\{ \theta^2 2^{3k/4}, \theta 2^{3k/8} \right\} \right)
        \,d\theta \\
        &\leq C \sigma 2^{-k/8}
    \end{align*}
    for some absolute constant $C > 0$\,.  This establishes \eqref{eq:k-hierarchy}.  Combining with \eqref{eq:chaining-telescope} completes the proof.
\end{proof}

\begin{proof}[Proof of \cref{lem:upper-bound-max-chaining}]
    The lemma is equivalent to \eqref{eq:chaining-new}, which follows from combining \cref{lem:chaining-sparse-part} and \cref{lem:chaining-gaussian-part} with \eqref{eq:chaining-split-scales}.
\end{proof}

\section*{Acknowledgments}
\noindent M.M. is supported in part by NSF CAREER grant DMS-2336788.  Both authors are supported in part by NSF grants DMS-2137623 and DMS-2246624. The authors thank Igor Pak for bringing this problem to their attention.

\bibliographystyle{abbrv}
\bibliography{refs}

\end{document}